\documentclass[a4paper,11pt]{article}

\usepackage[sort&compress,numbers]{natbib}

\usepackage[english]{babel}
\renewcommand{\thefootnote}{\alph{footnote}}

\usepackage{amsmath, amsthm, amsfonts, amssymb}

\usepackage{graphicx}
\usepackage[margin=1in]{geometry}

\usepackage{array, booktabs, makecell, hhline, diagbox}
\usepackage[table]{xcolor}

\usepackage{microtype}
\usepackage{color}
\usepackage{soul}
\usepackage{enumerate}
\usepackage{multirow}
\usepackage[hidelinks]{hyperref}
\usepackage[nameinlink]{cleveref}
\usepackage{thm-restate}
\usepackage{xfrac}
\usepackage{physics}
\usepackage[ruled,vlined,norelsize]{algorithm2e}
\usepackage{verbatim}
\usepackage{etoolbox}
\usepackage{float}
\usepackage{tikz}
\usetikzlibrary{positioning}
\usetikzlibrary{calc}
\usepackage{bm}

\makeatletter
\renewcommand*{\@fnsymbol}[1]{\ifcase#1\or1\else\@arabic{\numexpr#1-1\relax}\fi}
\makeatother

\newtheorem{theorem}{Theorem}[section]

\newtheorem{lemma}[theorem]{Lemma}

\newtheorem{conjecture}[theorem]{Conjecture}
\newtheorem{remark}[theorem]{Remark}
\newtheorem{proposition}[theorem]{Proposition}

\theoremstyle{remark}

\newcommand{\keywords}[1]{\vspace{0.5em}\noindent\textbf{Keywords:} #1}
\newcommand{\old}[1]{{{}}}

\def\C{\mathcal{C}}

\def\bfell{\bm{\ell}}

\def\R{\mathcal{R}}

\def\T{\mathcal{T}}

\def\longto{\longrightarrow}

\let\eps\varepsilon
\usepackage{xspace}

\def\Thanks#1{\gdef\thefootnote{\arabic{footnote}}\thanks{#1}}

\newenvironment{Proof}[1]{
  \par\vspace{1em}
  \noindent{\bf Proof#1. }\quad
}{
  \par
}

\title
{A cop--robber game on metric graphs}
\author{D.~Berend\Thanks{Institute for the Theory of Computing and Department of Mathematics, Ben-Gurion University, Beer Sheva 84105, Israel. E-mail: berend@bgu.ac.il}
\and
M. D. Boshernitzan\Thanks{Deceased; formerly of Department of Mathematics,
Rice University, Houston, TX 77251, USA}}

\begin{document}

\maketitle
\renewcommand{\thefootnote}{}


\begin{abstract} 
We study a variant of the classical cop--robber game played on compact metric graphs, where each edge is assigned a positive length and identified with a real interval of corresponding length. In this setting, both the cop and the robber move continuously along the edges, subject to upper bounds on their speeds. The cop has no knowledge of the robber’s location and must choose a continuous path through the graph that is guaranteed to intersect the robber’s trajectory at some point in time. We show that for every compact metric graph, there exists a constant \( s > 0 \) such that if the cop’s speed exceeds \( s \) times the robber’s speed, then the cop can guarantee capture.
\end{abstract}


\setcounter{section}{0}

\section{Introduction}
\setcounter{equation}{0}
\medskip

The cop--robber game, a pursuit--evasion problem on graphs, originated in
the late 1970s in the work of Quilliot~\cite{quilliot-thesis,quilliot}
and independently by Nowakowski and Winkler~\cite{nowakowski}. It was further
developed by Aigner and Fromme~\cite{aigner} and popularized through
connections to recreational mathematics, notably in the columns by Martin
Gardner~\cite{gardner1983,gardner1994}. Earlier conceptual foundations
of pursuit and evasion can be found in Ryll-Nardzewski’s formal
treatment~\cite{ryll1962} and Isaacs’ theory of differential
games~\cite{isaacs}. Over time, the game has inspired a wide range of
combinatorial, algorithmic, and geometric investigations.

Many variants of the game have been studied, involving multiple cops~\cite{bonato_nowakowski_book},
randomization~\cite{alspach}, restricted visibility~\cite{seager}, and partial
information~\cite{isler2008role}. A central quantity in these settings is
the \emph{cop number}, the minimum number of cops needed to guarantee capture.
Most of this theory is discrete: players occupy vertices and move along
edges in discrete time. In contrast, pursuit--evasion in Euclidean domains
and continuous environments has long been a subject of differential
games, yet formal study on \emph{metric graphs}, structures combining
combinatorial branching with continuous motion, remains limited.

In this paper we study a continuous-time cop--robber game on a compact
metric graph \(G\). Each edge \(e\) is assigned a positive length
\(L_e>0\) and identified with an interval \([0,L_e]\); the graph is
equipped with its intrinsic shortest-path metric. Both players move
continuously along edges and are constrained only by upper bounds on
their speeds. We normalize the robber's maximal speed to \(1\); the
cop's maximal speed is then a parameter \(s>0\).

A \emph{strategy} for the cop is a continuous \(s\)-Lipschitz path
\(f:[0,T]\to G\); a \emph{strategy} for the robber is a continuous
\(1\)-Lipschitz path \(g:[0,T]\to G\). We say that the cop's strategy
\(f\) is \emph{winning} if for every admissible robber strategy \(g\)
there exists \(t\in[0,T]\) with \(f(t)=g(t)\) (in which case we say the
cop \emph{captures} the robber by time \(T\)). Since the cop receives no information about the robber's position during
play, his strategy must be a fixed trajectory chosen in advance; in
particular, it cannot depend on the robber's motion. Our model does not
impose acceleration bounds: players may change direction
instantaneously, which is naturally represented by the Lipschitz
condition.

A basic observation is that capture is not always possible. For
example, on a cycle the robber can avoid capture indefinitely when
\(s\le 1\), whereas on a path any \(s>0\) suffices for eventual capture.

Clearly, if for some speed $s$ there exists a winning strategy, then the same strategy is winning for every greater speed. The {\it critical speed} for $G$ is the infimum of all speeds for which there exists a winning strategy, and is denoted by $s^*(G)$. More precisely, we will distinguish between two situations, depending on whether the infimum is achieved or not; if the infimum is $s_0$ and for $s_0$ itself there exists a winning strategy, we denote the critical speed by $s_0$, and otherwise, by $s_0^+$. (We mention in passing that we do not know of a case where the first situation occurs, except for a one-point $G$ for which $s^*=0$.) For example, referring to the two trivial examples above, the critical speed of a path is $0^+$ and that of a cyclic graph is $1^+$.

In Section~\ref{sec:main-results} we present our main results: the
finiteness of the critical speed on every compact metric graph
(Theorem~\ref{finiteness-of-critical-speed}), upper bounds for certain
families such as star graphs (Theorem~\ref{critical-speed-of-star}) and
comb graphs (Proposition~\ref{complex-graph-with-small-critical-speed}),
and structural theorems about the existence and properties of optimal
strategies (Theorems~\ref{optimal-strategy-exists} and~\ref{max-speed}).
Section~\ref{sec:proofs} contains all proofs and supporting lemmas.

\section{The Main Results}\label{sec:main-results}
\medskip

We work throughout with compact metric graphs. Loops and multiple edges
do not introduce additional generality: any metric multigraph can be
converted into an equivalent metric graph by subdividing parallel edges
and loops with intermediate vertices. Thus it suffices to restrict
attention to metric graphs in the usual sense.

Moreover, if the graph is disconnected, then no winning strategy exists.
We therefore assume throughout that the underlying metric graph is
connected.

Our first result shows that, on any metric graph, C can catch R if $s$ is large enough.

\begin{theorem} \label{finiteness-of-critical-speed}
The critical speed of every metric graph is finite.
\end{theorem}

For certain classes of graphs, we can derive upper bounds on the critical speed. In particular, our next result establishes such a bound for \emph{star graphs}.  
Recall that, for \( k \ge 2 \), the \(k\)-star graph \(S_k\) is the tree consisting of one internal (central) node connected to \(k\) leaves; see Figure~\ref{fig:star}.

\begin{figure}[H]
\centering
\begin{tikzpicture}[scale=1.2, every node/.style={circle, draw, fill=white, inner sep=2pt}]
  \node (c) at (0,0) {};
  \foreach \i in {1,...,5} {
    \node (l\i) at ({1.5*cos(72*(\i-1))},{1.5*sin(72*(\i-1))}) {};
    \draw (c) -- (l\i);
  }
\end{tikzpicture}
\caption{The $k$-star graph $S_k$ (for $k=5$): a tree with one central node and $k$ leaves.}
\label{fig:star}
\end{figure}
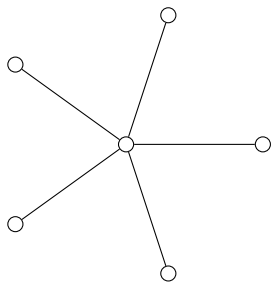

\begin{theorem} \label{critical-speed-of-star}
For every $k\ge 3$ and edge lengths $\bfell$:
$$s^*\left(S_k\right) \le (2k-3)^+.$$
\end{theorem}

As critical speeds are of either of the two forms $s$ or $s^+$ for some
non-negative $s$, we need to explain what the last inequality between
critical speeds means. Given numbers $s_2>s_1\ge 0$, we agree that
$s_1<s_1^+<s_2$.

In the cases $k=1,2$, the graph is a path, so
that $s^*\left(S_k\right)=0^+$.

\begin{conjecture}
For every $k\ge 3$ and edge lengths $\bfell$:
$$s^*\left(S_k\right) = (2k-3)^+.$$
\end{conjecture}

Theorem~\ref{critical-speed-of-star} suggests that the critical speed of a graph
may grow with its ``complexity''.  We use this term here in a purely intuitive
sense.  Informally, large critical speed appears to be associated with the
presence of vertices of high degree, and in particular with clusters of such
vertices that are close to one another.  We do not attempt to formalize this
intuition.

With this in mind, consider the following family of graphs.
Denote by $B_k$ the \emph{comb graph} consisting of a path of length $k$, with
each vertex on the path having a single leaf attached; see
Figure~\ref{fig:comb}.
(This graph is also known as a \emph{centipede graph}; see~\cite{wolframCentipedeGraph}.)

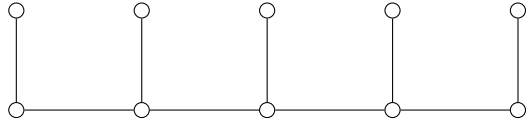
\begin{figure}[H]
\centering
\begin{tikzpicture}[scale=1.1, every node/.style={circle, draw, fill=white, inner sep=2pt}]
  \foreach \i in {0,...,4} {
    \node (p\i) at (\i*1.5, 0) {};
    \ifnum\i>0
      \pgfmathtruncatemacro{\j}{\i-1}
      \draw (p\j) -- (p\i);
    \fi
  }

  \foreach \i in {0,...,4} {
    \node (l\i) at (\i*1.5, 1.2) {};
    \draw (p\i) -- (l\i);
  }
\end{tikzpicture}
\caption{The graph \(B_k\) (for \(k=5\)): a path of \(k\) vertices, each with a leaf attached.}
\label{fig:comb}
\end{figure}

\begin{proposition} \label{complex-graph-with-small-critical-speed}
For $k\ge 3$, there exist distances $\bfell$ on $B_k$, such that
$s^*(B_k,\bfell) \le 3^+$.
\end{proposition}

An \emph{optimal strategy} for the cop is one that guarantees capture of the robber in the least possible time, regardless of the robber’s actions.

\begin{theorem} \label{optimal-strategy-exists}
Every metric graph admits an optimal strategy.
\end{theorem}

We may express the theorem alternatively as follows. Denote by $\T$ the set of all $T > 0$ for which there exists a winning strategy $f : [0, T] \to X_G$. Clearly, if $T_1 \in \T$ and $T_2 > T_1$, then $T_2 \in \T$, so $\T$ is either of the form $[T_0, \infty)$ or $(T_0, \infty)$. The theorem asserts that the former always holds.

\begin{remark}
It is possible that an optimal strategy of \(C\) will intersect every strategy of \(R\) strictly before time \(T_0\). For example, consider the circle of length 1 with \(s=2\). An optimal strategy for \(C\) is to move clockwise (say) at full speed. Even if \(R\) starts just behind \(C\) (relative to \(C\)'s direction), capture occurs strictly before \(T_0 = 1\). However, no strategy guarantees capture uniformly by time \(1 - \varepsilon\) for any \(\varepsilon > 0\).
\end{remark}

Intuitively, there is no reason for C ever to move at less than the maximal speed~$s$. Since C’s strategy may involve (and, as
will be seen in the proof of Theorem
\ref{critical-speed-of-star}, typically does) infinitely many changes in direction, it is not immediately clear how to formalize this notion of moving at full speed. We adopt the interpretation that C moves at maximal speed if the total variation $V_f$ of the strategy $f : [0,T] \to X_G$ satisfies $V_f = sT$. Here, the \emph{total variation} of $f$ is defined by
\[
V_f := \sup \left\{ \sum_{i=1}^n d_G(f(t_{i-1}), f(t_{i})) \,:\, 0 = t_0 < t_1 < \cdots < t_n = T \right\},
\]
where \( d_G \) denotes the intrinsic (shortest-path) metric on the graph \( G \). The
following theorem establishes a strong version of
that intuition.

\begin{theorem} \label{max-speed}
In every optimal strategy for a metric graph, C moves throughout the hunt at maximal speed.
\end{theorem}

\begin{remark}
From the theorem it follows that in any optimal strategy, the cop spends zero total time at each point. 
This does not preclude the cop from being at a given vertex at an uncountable set of times (for example, a Cantor set). 
It is an open question whether every optimal strategy can be realized in a more regular form: 
\begin{enumerate}
\item the cop is at each vertex only countably many times, and
\item on each connected component of the complement of these times (the interim intervals), the cop either moves at top speed from one vertex to another or moves at top speed from a vertex along some edge and back.
\end{enumerate}
\end{remark}

\section{Structural Lemmas}\label{sec:structural-lemmas}
\medskip

\begin{lemma} \label{shorten-edge}
Let \(G\) be a metric graph, and let \(G'\) be the graph obtained from \(G\)
by shortening an edge that is incident to a leaf vertex
(see Figure \ref{fig:shorten-edge}).
Then 
\[
s^*(G') \le s^*(G).
\]
\end{lemma}

\begin{figure}[H]
\centering
\begin{tikzpicture}[scale=1, every node/.style={circle, draw, fill=white, inner sep=2pt}]
  \draw[thick] (-1,-1.5) rectangle (3.5,1.7);
  \node[draw=none, fill=none] at (1.25,2.0) {$G$};

  \node (a) at (0,0) [label={[label distance=1mm]below:{$v_5$}}] {};
  \node (b) at (0.7,0.7) [label=above:{$v_3$}] {};
  \node (c) at (0.7,-0.7) [label=below:{$v_4$}] {};
  \node (d) at (1.4,0) [label={[label distance=1mm]below:{$v_1$}}] {};
  \node (e) at (2.8,0) [label={[label distance=1mm]below:{$v_2$}}] {};

  \draw (a) -- (b);
  \draw (a) -- (c);
  \draw (a) -- (d);
  \draw (b) -- (d);
  \draw (d) -- (e);

  \draw[thick,->,>=stealth,line width=1pt] (3.7,0) -- (4.8,0);

  \draw[thick] (5,-1.5) rectangle (8.5,1.7);
  \node[draw=none, fill=none] at (6.75,2.0) {$G'$};

\node (a') at (5.8,0)  [label={[label distance=2.5mm]below:{$v_5$}}] {};
\node (b') at (6.5,0.7) [label=above:{$v_3$}] {};
\node (c') at (6.5,-0.7) [label=below:{$v_4$}] {};
\node (d') at (7.2,0)  [label={[label distance=2.5mm]below:{$v_1$}}] {};
\node (e') at (7.7,0)  [label={[label distance=1.5mm]below:{$v_2'$}}, inner sep=1.5pt] {};

\draw (a') -- (b');
\draw (a') -- (c');
\draw (a') -- (d');
\draw (b') -- (d');
\draw (d') -- (e');

\end{tikzpicture}
\caption{Graphs \(G\) and \(G'\), where \(G'\) is obtained from \(G\) by shortening the edge from \(v_1\) to the leaf \(v_2\).}
\label{fig:shorten-edge}
\end{figure}
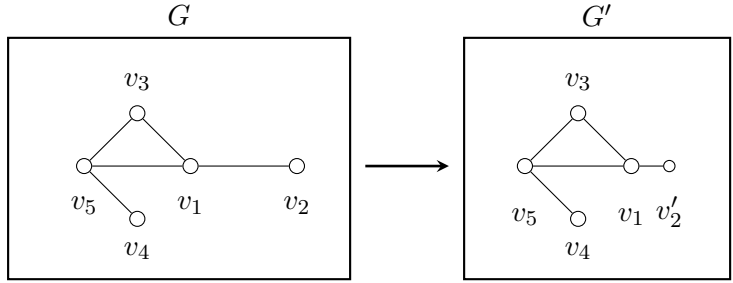

\begin{remark}
At first glance, it may seem plausible that the lemma remains valid when \emph{any} edge is shortened, not necessarily one adjacent to a leaf. However, if this were true, then by Lemma~\ref{constant-factor} below, the critical speed would depend only on the underlying graph structure (i.e., the isomorphism class of the graph) and not on the edge lengths. We find this unlikely.
\end{remark}

\begin{Proof}{\ of Lemma \ref{shorten-edge}}
Let $e$ be an edge in $G$ connecting a vertex $v_1$ and a leaf $v_2$. Without loss of generality, assume that the coordinate $x_e=0$ at $v_1$ and $x_e = L_e$ at $v_2$. Let $G'$ be the graph obtained from $G$ by shortening $e$ to length $L_e' < L_e$, replacing $v_2$ with a new leaf $v_2'$. Let $f : [0,T] \to G$ be a winning strategy for C on $G$.

We define a new strategy $f'$ for $G'$ that mirrors $f$ except when $f(t)$ lies beyond the shortened edge. Formally, define $f'$ by:
\[
f'(t)=
\begin{cases}
f(t), &  f(t) = (\tilde{e},r) \text{ with } \tilde{e} \ne e, \\
f(t), &  f(t) = (e,r) \text{ with } r \le L_e', \\
(e, L_e'), &  f(t) = (e,r) \text{ with } r > L_e'.
\end{cases}
\]

We claim that $f'$ is a winning strategy for $G'$. To see this, let $g'$ be any strategy of R in $G'$. Clearly, $g'$ is also a valid strategy in $G$. (Here we have used the fact that $v_2$ is a leaf.) Because $f$ is a winning strategy in $G$, there exists some $t \in [0,T]$ such that $f(t) = g'(t)$. But since $g'(t) \in G'$, and $f'(t)$ agrees with $f(t)$ wherever $f(t) \in G'$, it follows that $f'(t) = g'(t)$. Thus, C catches R in $G'$ using strategy $f'$.

This construction is illustrated in Figure~\ref{fig:shorten-edge}, which shows how $f(t)$ is projected onto the shortened edge to yield $f'(t)$.
\end{Proof}

\begin{lemma} \label{constant-factor}
Let $G$ be an arbitrary metric graph, and $G'$ be the graph obtained from it by multiplying all edge lengths by some constant $c>0$. Then $s^*(G')=s^*(G)$.
\end{lemma}

\begin{proof}
It suffices to show that, if for some $T$ the function $f:[0,T]\to G$
is a winning strategy for $G$ with Lipschitz constant $s$, then there
exists a winning strategy $f':[0,T']\to G'$ for $G'$. In fact, take
$T'=cT$, and define $f'$ as follows. For $t\ge 0$, put $(e,x)=f(t/c)$, and
let $f'(t)=(e,cx)$. It is easy to verify that $f'$ satisfies the required
property.
\end{proof}

This invariance under uniform scaling justifies our assumption in the introduction of fixing the robber's speed to 1: only the ratio of speeds matters.

\section{Stars}
\label{sec:stars}
\setcounter{equation}{0}
\medskip

It will be more convenient to prove Theorem~\ref{critical-speed-of-star} first, as the proof holds the key to the proof of Theorem~\ref{finiteness-of-critical-speed} in a simpler, more structured setting.

Let \( G \) be a star with center vertex \( O \), and leaves \( v_1, v_2, \ldots, v_k \), connected to \( O \) via edges \( Ov_1, Ov_2, \ldots, Ov_k \) of lengths \( \ell_1, \ldots, \ell_k \), respectively. (Refer again to Figure~\ref{fig:star}.) 
We have to show that $s^*\left(S_k\right)\le (2k-3)^+$, namely, if $s=2k-3+\eps$
for some $\eps>0$, then there exists a winning strategy.

The first step of the cop's strategy is straightforward: the cop travels to the center \( O \), proceeds along the edge \( Ov_k \) to its endpoint, and then returns to \( O \). If the robber is encountered along the way, the game ends. Otherwise, the cop has verified that the robber is not on \( Ov_k \), and may treat this edge as cleared. An edge \( Ov_i \), or more generally, any segment \( [A,B] \) on an edge, is considered \emph{cleared} if the robber is not present at any point $x \in [A, B]$.

Figure~\ref{fig:one-edge-cleaning} illustrates the state of the game before and after the cop’s initial move. All edges are initially unclean, represented by red dashed lines (left part of the figure). After the cop visits the edge \( Ov_k \), it becomes clean, indicated by a solid green line (right part).

\begin{figure}[ht]
\centering
\begin{tikzpicture}[scale=2]

\def\xshift{5cm}

\def\angles{{0,72,144,216,288}}

\node[fill=black, circle, inner sep=1.5pt] (O) at (0,0) {};
\node[fill=black, circle, inner sep=1.5pt] (OR) at (\xshift,0) {};

\foreach \i [evaluate=\i as \angle using {\angles[\i-1]}] in {1,...,5} {
    \node[circle, draw=black, fill=white, inner sep=1.5pt] (v\i) at (\angle:0.9cm) {};
    \draw[thick, dashed, red] (O) -- (v\i);
}

\foreach \i [evaluate=\i as \angle using {\angles[\i-1]}] in {1,...,5} {
    \node[circle, draw=black, fill=white, inner sep=1.5pt] (v\i R) at ([xshift=\xshift] \angle:0.9cm) {};
    \ifnum\i=5
        \draw[thick, green!70!black] (OR) -- (v\i R); 
    \else
        \draw[thick, dashed, red] (OR) -- (v\i R); 
    \fi
}

\node[below] at (O) {$O$};
\node[below] at (OR) {$O$};

\node[above right] at (v1) {$v_1$};
\node[above left] at (v2) {$v_2$};
\node[left] at (v3) {$v_3$};
\node[below left] at (v4) {$v_4$};
\node[below right] at (v5) {$v_5$};

\node[above right] at (v1R) {$v_1$};
\node[above left] at (v2R) {$v_2$};
\node[left] at (v3R) {$v_3$};
\node[below left] at (v4R) {$v_4$};
\node[below right] at (v5R) {$v_5$};

\end{tikzpicture}
\caption{A single edge is cleaned in Phase 1.}
\label{fig:one-edge-cleaning}
\end{figure}
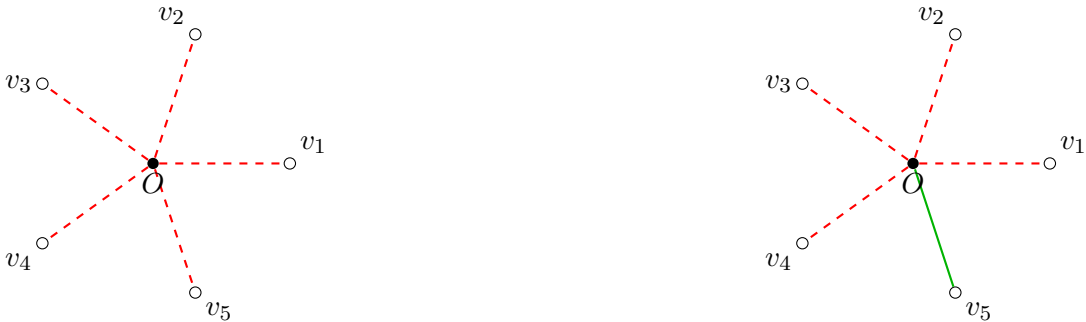

At this point, however, a difficulty arises. We know that the robber is located somewhere on one of the other edges \( Ov_1, \ldots, Ov_{k-1} \), and hence must be at a positive distance from \( O \). However, this distance is unknown and could be arbitrarily small. Therefore, even though the robber cannot be located exactly at \( O \), we cannot assert that any initial segment of these arms has been cleared.

Moreover, any attempt to explore one of the remaining arms — say, \( Ov_1 \) — introduces risk: during the time the cop is away from \( O \), the robber might exploit the opportunity to enter the already-cleared edge \( Ov_k \). Upon the cop’s return, there is no longer any guarantee that \( Ov_k \) remains clean.

The main challenge is thus to explore unknown arms while preserving control over cleared regions. The following lemma provides a key step toward resolving this.

Denote by $\lambda$ the unique positive root of the polynomial
$x^{k-2}+x^{k-3}+\ldots+x-\frac{s-1}{2}$. Since $s>2k-3$, we have
$\lambda>1$.

\begin{lemma}
Let \(s > 2k - 3\). Suppose that the cop C is at the vertex \( O \), and:
\begin{itemize}
\item The edge \(Ov_k\) is fully cleared;
\item For some \(d > 0\), the segment \([O,u_i]\) of edge \(Ov_i\) is cleared for each \(i = 2, \dotsc, k-1\), where
\[ 
|Ou_2| = d, \quad |Ou_3| = (\lambda + 1)d, \quad \dotsc, \quad |Ou_{k-1}| = (\lambda^{k-3} + \cdots + 1)d.
\]
\end{itemize}
Suppose that C now moves at speed \(s\) for \(d/2\) time units along \(Ov_1\) toward \(v_1\), and then for \(d/2\) time units back to \(O\). Then:
\begin{itemize}
\item Either C has intercepted R during this time interval;
\item Or R has remained on the same edge as at the beginning.
\end{itemize}
Moreover, in the second case, the edges \(Ov_i\) for \(i = 1, 3, 4, \dotsc, k - 1\) are cleared up to points \(u'_1, u'_3, u'_4, \dotsc, u'_{k-1}\), satisfying:
\begin{itemize}
\item \( |Ou'_1| = \min\{\lambda \cdot |Ou_{k-1}|,\ |Ov_1|\} \);
\item \( |Ou'_i| = \lambda \cdot |Ou_{i-1}| \) for \(i = 3, 4, \dotsc, k-1\).
\end{itemize}
\end{lemma}

\begin{remark}
    Note that our assumptions and conclusions ``miss'' some edges. In the beginning, R may be on \( Ov_1 \), arbitrarily close to \( O \). In the end, R may be on \( Ov_2 \), arbitrarily close to \( O \). The lemma also implies that \( Ov_k \) is clean at the end of the iteration.
\end{remark}

\begin{remark} \label{Ov1 too short}
    The lemma does not assume that the cop is aware of which portions of edges are cleared. Rather, it states that if certain segments are in fact clear initially, then further segments are guaranteed to be cleared after the described motion. This distinction will be important in the iterative construction later.

If \( |Ov_1| < sd/2, \) then C reaches \( v_1 \) within less than \( d/2 \) time units. In this case, the way back starts earlier, and the lower bounds on R's distance from \( O \), depending on his location at the beginning of the iteration, are even larger than stated in the lemma.
\end{remark}

Figure~\ref{fig:one-iteration-expansion} visualizes the initial and resulting configurations described in the lemma: cleared segments and key points are marked according to the lemma's assumptions and conclusions.

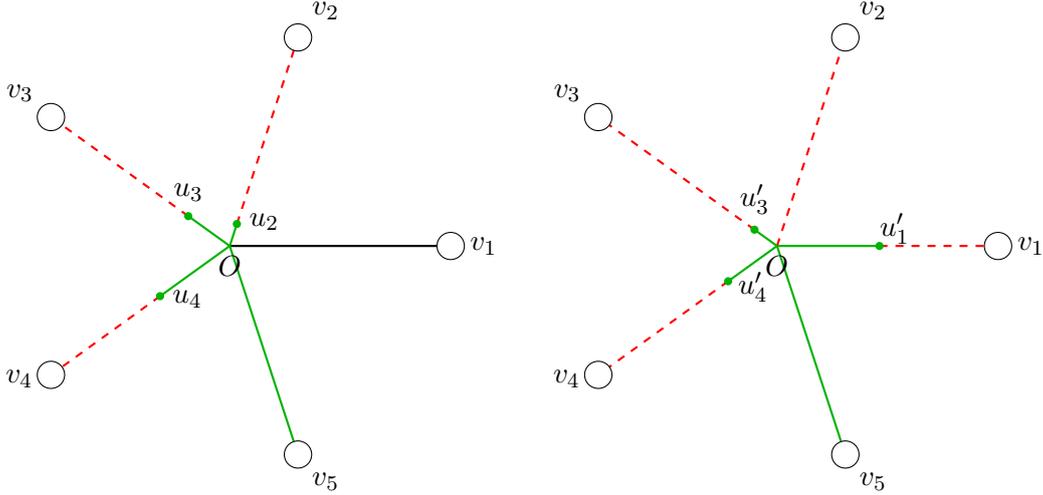
\begin{figure}
\centering
\begin{tikzpicture}[scale=1.8]
\def\rnode{0.18}
\def\edgelen{1.7}
\def\labelfac{1.03}
\def\vertexfac{0.95}

\coordinate (Oleft) at (-2,0);
\coordinate (Oright) at (2,0);


\coordinate (Lv1end) at ($(Oleft)+(0:\edgelen)$);
\coordinate (Lv2end) at ($(Oleft)+(72:\edgelen)$);
\coordinate (Lv3end) at ($(Oleft)+(144:\edgelen)$);
\coordinate (Lv4end) at ($(Oleft)+(216:\edgelen)$);
\coordinate (Lv5end) at ($(Oleft)+(288:\edgelen)$);

\coordinate (u2) at ($(Oleft)!0.1!(Lv2end)$);
\coordinate (u3) at ($(Oleft)!0.22!(Lv3end)$);
\coordinate (u4) at ($(Oleft)!0.37!(Lv4end)$);

\coordinate (Lv1) at ($(Oleft)+(0:{\vertexfac*\edgelen})$);
\coordinate (Lv2) at ($(Oleft)+(72:{\vertexfac*\edgelen})$);
\coordinate (Lv3) at ($(Oleft)+(144:{\vertexfac*\edgelen})$);
\coordinate (Lv4) at ($(Oleft)+(216:{\vertexfac*\edgelen})$);
\coordinate (Lv5) at ($(Oleft)+(288:{\vertexfac*\edgelen})$);

\draw[thick] (Oleft) -- ($(Oleft)+(0:{\edgelen-\rnode})$);
\draw[thick, green!70!black] (Oleft) -- (u2);
\draw[thick, red, dashed] (u2) -- ($(Oleft)+(72:{\edgelen-\rnode})$);
\draw[thick, green!70!black] (Oleft) -- (u3);
\draw[thick, red, dashed] (u3) -- ($(Oleft)+(144:{\edgelen-\rnode})$);
\draw[thick, green!70!black] (Oleft) -- (u4);
\draw[thick, red, dashed] (u4) -- ($(Oleft)+(216:{\edgelen-\rnode})$);
\draw[thick, green!70!black] (Oleft) -- ($(Oleft)+(288:{\edgelen-\rnode})$);

\foreach \v in {Lv1,Lv2,Lv3,Lv4,Lv5}
    \node[circle, draw=black, fill=white, minimum size=2*\rnode cm, inner sep=0pt] at (\v) {};

\node[below] at (Oleft) {$O$};
\node[right=1pt] at ($(Oleft)!{\labelfac}!(Lv1)$) {$v_1$};
\node[above right=1pt] at ($(Oleft)!{\labelfac}!(Lv2)$) {$v_2$};
\node[above left=1pt] at ($(Oleft)!{\labelfac}!(Lv3)$) {$v_3$};
\node[left=1pt] at ($(Oleft)!{\labelfac}!(Lv4)$) {$v_4$};
\node[below right=1pt] at ($(Oleft)!{\labelfac}!(Lv5)$) {$v_5$};

\fill[green!70!black] (u2) circle (0.03cm);
\fill[green!70!black] (u3) circle (0.03cm);
\fill[green!70!black] (u4) circle (0.03cm);

\node at ($(u2)+(0.2,0.02)$) {$u_2$};
\node[above=2pt] at (u3) {$u_3$};
\node at ($(u4)+(0.2,-0.02)$) {$u_4$};


\coordinate (Rv1end) at ($(Oright)+(0:\edgelen)$);
\coordinate (Rv2end) at ($(Oright)+(72:\edgelen)$);
\coordinate (Rv3end) at ($(Oright)+(144:\edgelen)$);
\coordinate (Rv4end) at ($(Oright)+(216:\edgelen)$);
\coordinate (Rv5end) at ($(Oright)+(288:\edgelen)$);

\coordinate (ur1) at ($(Oright)!0.44!(Rv1end)$);
\coordinate (ur3) at ($(Oright)!0.12!(Rv3end)$);
\coordinate (ur4) at ($(Oright)!0.26!(Rv4end)$);

\coordinate (Rv1) at ($(Oright)+(0:{\vertexfac*\edgelen})$);
\coordinate (Rv2) at ($(Oright)+(72:{\vertexfac*\edgelen})$);
\coordinate (Rv3) at ($(Oright)+(144:{\vertexfac*\edgelen})$);
\coordinate (Rv4) at ($(Oright)+(216:{\vertexfac*\edgelen})$);
\coordinate (Rv5) at ($(Oright)+(288:{\vertexfac*\edgelen})$);

\draw[thick, green!70!black] (Oright) -- (ur1);
\draw[thick, red, dashed] (ur1) -- ($(Oright)+(0:{\edgelen-\rnode})$);
\draw[thick, red, dashed] (Oright) -- ($(Oright)+(72:{\edgelen-\rnode})$);
\draw[thick, green!70!black] (Oright) -- (ur3);
\draw[thick, red, dashed] (ur3) -- ($(Oright)+(144:{\edgelen-\rnode})$);
\draw[thick, green!70!black] (Oright) -- (ur4);
\draw[thick, red, dashed] (ur4) -- ($(Oright)+(216:{\edgelen-\rnode})$);
\draw[thick, green!70!black] (Oright) -- ($(Oright)+(288:{\edgelen-\rnode})$);

\foreach \v in {Rv1,Rv2,Rv3,Rv4,Rv5}
    \node[circle, draw=black, fill=white, minimum size=2*\rnode cm, inner sep=0pt] at (\v) {};

\node[below] at (Oright) {$O$};
\node[right=1pt] at ($(Oright)!{\labelfac}!(Rv1)$) {$v_1$};
\node[above right=1pt] at ($(Oright)!{\labelfac}!(Rv2)$) {$v_2$};
\node[above left=1pt] at ($(Oright)!{\labelfac}!(Rv3)$) {$v_3$};
\node[left=1pt] at ($(Oright)!{\labelfac}!(Rv4)$) {$v_4$};
\node[below right=1pt] at ($(Oright)!{\labelfac}!(Rv5)$) {$v_5$};

\fill[green!70!black] (ur1) circle (0.03cm);
\fill[green!70!black] (ur3) circle (0.03cm);
\fill[green!70!black] (ur4) circle (0.03cm);

\node at ($(ur1)+(0.11,0.13)$) {$u'_1$};     
\node[above=2pt] at (ur3) {$u'_3$};
\node at ($(ur4)+(0.18,-0.04)$) {$u'_4$};
\end{tikzpicture}
\caption{Expansion of free zones during one iteration.}
\label{fig:one-iteration-expansion}
\end{figure}

The proof of the lemma is straightforward. Let the cop's position at time $t$ be denoted by $f(t) \in X_G$. Suppose the current iteration begins at time $t_0$ with $f(t_0)=O$. The cop moves along $Ov_1$ towards $v_1$ at speed $s$ for $d/2$ time units and then returns. As C devotes $d$ time units to explore $Ov_1$, it is possible for R to reduce his distance from $O$ by $d$ distance units. Thus, if at the beginning of the iteration R was:
\begin{itemize}
    \item on $Ov_2$ -- his final distance from $O$ will be guaranteed to be larger only than $d-d=0$ by the end;
    \item on $Ov_3$ -- his final distance from $O$ will be guaranteed to be larger only than $(\lambda+1)d-d=\lambda d=\lambda\cdot |Ou_2|$ by the end;
    \item ...
    \item on $Ov_{k-1}$ -- his final distance from $O$ will be larger only than
    \[
    (\lambda^{k-3}+ \cdots +\lambda+1)d-d=(\lambda^{k-3} + \cdots +\lambda) d = \lambda |Ou_{k-2}|
    \]
    by the end.
\end{itemize}
In all these cases, R's distance from $O$ may have gone down during the iteration. However, if R was on $Ov_1$ and has not been caught during the time C entered this edge, then after $d/2$ time units he was at a distance larger than $ds/2$ from $O$, and by the end of the iteration at a distance larger than
\[
\frac{d}{2}\cdot s - \frac{d}{2}\cdot 1 =
\frac{s-1}{2}\cdot d = \left(\lambda^{k-2} + \lambda^{k-3} + \cdots + \lambda \right) d 
=\lambda |Ou_{k-1}|.
\]
As mentioned in Remark~\ref{Ov1 too short}, if R gets to \( v_1 \) within less than \( d/2 \) time units, then the edge \( Ov_1 \) will be clear by the end of the iteration. This proves the lemma.

We return to the proof of the theorem. Suppose now that at some point the cop is located at $O$, and we
somehow know that each of the edges
$Ov_2, Ov_3,\ldots, Ov_{k-1}$ is free from $O$ up to some point at a
positive distance from $O$ (thus, no assumption is made regarding $Ov_1$;
recall that $Ov_k$ is completely free). Specifically, assume that, for a
certain $d>0$, the edge $Ov_2$ is free up to $u_2$ where $|Ou_2|=d$, the
edge $Ov_3$ is free up to $u_3$ where $|Ou_3|=(\lambda+1) d$, and so forth, up
to the edge $Ov_{k-1}$, known to be free up to $u_{k-1}$ where
$|Ou_{k-1}|=(\lambda^{k-3}+\lambda^{k-4}+\ldots+1) d$.
Send the cop from $O$ towards $v_1$ at full speed for $d/2$ time
units, so that he gets up to some point $u'_1$ along this edge whose
distance from $O$ is $sd/2$, and then return to $O$ at the same speed.
We claim first that the
robber could not have used the cop's absence from $O$ to get there
(and from there perhaps to another edge than the one he was on at the
beginning of the cop's travel). Indeed, if at the beginning of the
stage the robber was somewhere along $Ov_2$, then within the $d$ time
units the cop spent along $Ov_1$ the robber could have reduced his
distance from $O$ by at most $d$, and hence could not have reached $O$.
Since on the other edges the free zones were larger, the situation is even
better if he was located along another edge at the beginning of the stage.
Moreover, while we have no positive lower bound on his distance from $O$ if
he is along $Ov_2$, we do have such bounds regarding the edges $Ov_1$ and
$Ov_3, Ov_4,\ldots, Ov_{k-1}$. Indeed, if he is along $Ov_1$ (and has not
been caught there during this stage), then he was at a distance exceeding
$sd/2$ from $O$ at the time the cop was at $u'_1$, so that by the end
of the stage his distance is larger than $(s-1)d/2$. If he was along $Ov_i$
for some $3\le i\le k-1$, then since his initial distance from $O$ was
larger than $(\lambda^{i-2}+\lambda^{i-3}+\ldots+1) d$, and it could have
reduced by at most $d$, it is at the end of this stage larger than
$(\lambda^{i-2}+\lambda^{i-3}+\ldots+\lambda) d$. Thus, comparing the sizes
of the free zones now along the edges $Ov_3,Ov_4,\ldots,Ov_{k-1},Ov_1$, to
the sizes of the free zones at the beginning of the stage along the edges
$Ov_2,Ov_3,\ldots,Ov_{k-1}$, respectively, we observe that all sizes
increased by a factor of $\lambda$. Consequently, within a finite number of
stages all edges but one will be free. We now send the cop to the
endpoint of this edge and thus ensure catching the robber. Note that
whenever in the course of this process the cop gets to the end of an
edge while
there are still some edges not totally freed, he just returns to $O$.

We still need to explain how we get to the situation whereby we know that
there are $k-2$ edges, not including $Ov_k$, with a free zone of positive
length. 
The idea is that, in fact, after the cop has finished histrip from
$v_k$ to $O$, all other edges do have some intervals of positive length
where the robber is not located; the only problem is that we do not know
how large these intervals are. Hence we do the procedure described above
``starting'' with infinitely small intervals. More precisely, we have seen
above that the cop spends along each edge $\lambda$ times as much
time as he did along the edge he has visited just prior to this edge. Thus
the cop will go along the edges $Ov_1, Ov_2,\ldots,Ov_{k-1}$
periodically, for time durations of $\ldots,
\lambda^{-3},\lambda^{-2},\lambda^{-1},1$. Formally (measuring time only
after the cop has finished the trip from $v_k$ to $O$), for each
positive integer $m$ and each integer $i$ between $1$ and $k-1$, the
cop spends half the time interval
$$I_{m,i}=[\lambda^{-m(k-1)+i}\cdot\frac{1}{\lambda-1},
           \lambda^{-m(k-1)+i}\cdot\frac{\lambda}{\lambda-1}]$$
to travel from $O$ towards $v_i$ at full speed and the second half to
return.

At the time these trips towards the $v_i$'s, $1\le i\le k-1$, and
backwards begin, the
robber must be along one of the edges $Ov_i$ at some positive distance
from $O$. Consequently, for sufficiently large $m$, it is the case that by
time $\lambda^{-m(k-1)}\cdot\frac{\lambda}{\lambda-1}$ the robber cannot
possibly get to the vertex $O$. Moreover, if he was initially along $Ov_i$, he cannot
get by that time closer to $O$ than $\lambda^{-m(k-1)+i}$. Hence the
suggested path is winning.

\section{Other Proofs} \label{sec:proofs}

\begin{Proof}{\ of Theorem \ref{finiteness-of-critical-speed}}
Let $(G,\boldsymbol{\ell})$ be a finite connected metric graph, where each edge
$e$ has length $\ell_e>0$. Write
\[
\Lambda := \sum_{e\in E(G)} \ell_e
\]
for the total length of the graph.

It is a classical fact that from any starting point $x\in G$ there exists a
walk of total length at most $2\Lambda$ that traverses every edge of $G$.  This
follows from the standard ``double-tree'' argument: choosing any spanning tree
$T$ of $G$, a depth–first traversal of $T$ walks each tree-edge at most twice,
and any non-tree edge can similarly be incorporated through a detour traversed
twice.  Hence the total length of such a walk never exceeds $2\Lambda$.  (See,
e.g., \cite[§1.8]{Diestel} for background on Euler tours and double-tree
walks.) Consequently, we may take
\( L = 2\Lambda,\)
so that, regardless of the starting point, there always exists a walk of length
at most $L$ that traverses the entire graph.

Employing the idea of the construction in the proof of Theorem
\ref{critical-speed-of-star}, we see that, starting from an arbitrary
vertex $v$, we can guarantee after some finite time $T_v$ that the
robber is not within a distance of some $\eps=\eps_v$ from $v$. Indeed,
we mimic the construction there, but only as long as the cop does not
get to any vertex neighboring $v$. Note that, if the degree of $v$ is $k$,
then a speed just a little larger than $2k-3$ will not suffice. Indeed, the
principal part of the process in the proof for star graphs started after
one edge was known to be free, which we cannot ensure now. Also, in that
proof we repeatedly had an edge for which there was no lower bound on the
size of the safe zone, whereas here need lower bounds on the distance from $v$ for all edges incident to this vertex. However, a speed larger than $2k+1$ will certainly
do.

First design for the cop a path of length at most $L$ traversing the
graph. When walking along this path, each time he gets to a vertex $v$
for the first time, he secures as above a region of radius $\eps_v$ around
$v$. The time required for this is at most $L/s+\sum_{v\in V} T_v$.
Clearly, this time may be made arbitrarily small by making $s$ sufficiently
large. In particular, we may ensure that the robber could not get to any
vertex between the time the cop started securing its vicinity and the
time the process ended.

Thus, if $s$ selected to be sufficiently large, by the end of the above
process we can get a lower bound of the distance of the robber from the
nearest vertex.

Now send the cop again along a path of length at most $L$ traversing
the graph. If $s$ is sufficiently large, then the robber will not be
able to reach any vertex during this stage, and thus the cop is bound
to catch him. This proves the theorem.
\end{Proof}

\begin{Proof}{\ of Proposition~\ref{complex-graph-with-small-critical-speed}}
Fix vertex labels as shown in Figure~\ref{fig:comb-labeled}.

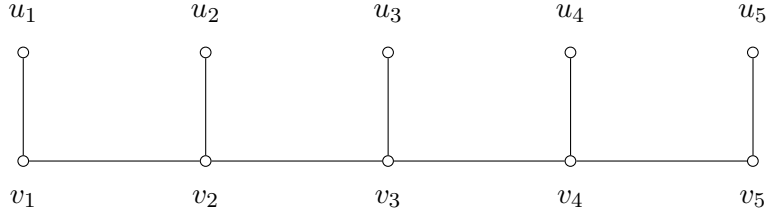
\begin{figure}[ht]
\centering
\begin{tikzpicture}[scale=1.2, every node/.style={circle, draw, fill=white, inner sep=0pt, minimum size=4pt}]
    \node (v1) at (0,0) {};
    \node (v2) at (2,0) {};
    \node (v3) at (4,0) {};
    \node (v4) at (6,0) {};
    \node (v5) at (8,0) {};
    
    \node[below=6pt of v1, draw=none, fill=none] {$v_1$};
    \node[below=6pt of v2, draw=none, fill=none] {$v_2$};
    \node[below=6pt of v3, draw=none, fill=none] {$v_3$};
    \node[below=6pt of v4, draw=none, fill=none] {$v_4$};
    \node[below=6pt of v5, draw=none, fill=none] {$v_5$};
    
    \draw (v1) -- (v2) -- (v3) -- (v4) -- (v5);
    
    \node (u1) at (0,1.2) {};
    \node (u2) at (2,1.2) {};
    \node (u3) at (4,1.2) {};
    \node (u4) at (6,1.2) {};
    \node (u5) at (8,1.2) {};
    
    \node[above=6pt of u1, draw=none, fill=none] {$u_1$};
    \node[above=6pt of u2, draw=none, fill=none] {$u_2$};
    \node[above=6pt of u3, draw=none, fill=none] {$u_3$};
    \node[above=6pt of u4, draw=none, fill=none] {$u_4$};
    \node[above=6pt of u5, draw=none, fill=none] {$u_5$};
    
    \draw (v1) -- (u1);
    \draw (v2) -- (u2);
    \draw (v3) -- (u3);
    \draw (v4) -- (u4);
    \draw (v5) -- (u5);
\end{tikzpicture}
\caption{Comb graph $B_5$ with backbone vertices $v_1,\dots,v_5$ and leaves $u_1,\dots,u_5$.}
\label{fig:comb-labeled}
\end{figure}

Let $s>3$. The cop goes first to $u_1$ and then to $v_2$ via $v_1$. The path from $u_1$ to $v_2$ is now clear. Upon reaching $v_2$, the cop faces exactly the same local situation as in the
$3$-star: one incident edge ($v_2v_1$) is known to be clear, while the two edges
$v_2u_2$ and $v_2v_3$ are suspected. Since C's speed exceeds 3, as in the star
case, the cop performs an infinite schedule of geometrically increasing
excursions into the two suspected edges. Since both suspected edges have the same length, he can clear $v_2u_2$ without giving R a chance to sneak from $v_2v_3$ to $v_1v_2$. After clearing $v_2u_2$, he goes to $v_3$. He continues similarly, performing $3$-star clearings at each $v_i$.

Therefore the cop has a winning strategy for every $s>3$, and hence
$s^*(B_k,\boldsymbol{\ell}) \le 3^+$.
\end{Proof}

\begin{Proof}{\ of Theorem\ \ref{optimal-strategy-exists}}
Denote $T_0=\inf\T$. Then there exists a sequence
$(f_n)_{n=1}^\infty$ of winning strategies,
$$f_n:[0,T_0+\eps_n]\longto X_G, \qquad n=1,2,\ldots,$$
where $\eps_n>0$ for each $n$ and $\eps_n\xrightarrow[n\to\infty]{} 0.$ Since all the $f_n$ are $s$-Lipschitz, the sequence has a
pointwise convergent subsequence (defined on $[0,T_0]$).
Replacing the original sequence by this subsequence, we may
assume that $f_n\xrightarrow[n\to\infty]{} f$ pointwise. Clearly,
$f$ is also $s$-Lipschitz. It remains to show that $f$ forms a
winning strategy.

Let $g$ be a strategy of the robber. Since the $f_n$ are
winning, for each $n$ there exists a point
$t_n\in[0,T_0+\eps_n]$ such that
\begin{equation}\label{intersection-point}
f(t_n)=g(t_n).
\end{equation}
Passing again to a subsequence, we may assume that the
sequence $(t_n)_{n=1}^\infty$ is convergent, say
$t_n\xrightarrow[n\to\infty]{} t_0$. Note that $t\in [0,T_0]$. Since
the $f_n$ are uniformly Lipschitz, we have
$$f_n(t_n)\xrightarrow[n\to\infty]{} f(t_0).$$
Since $g$ is continuous:
$$g(t_n)\xrightarrow[n\to\infty]{} g(t_0).$$
By (\ref{intersection-point}), the last two convergence
properties imply that $f(t_0)=g(t_0)$, which completes the
proof.
\end{Proof}

\begin{Proof}{\ of Theorem\ \ref{max-speed}}
Let $T_0=\min\T$, and suppose there exists an optimal strategy
$f$ whose total variation is $V_f<sT_0$. For $t\in [0,T_0]$,
denote by $V_f(t)$ the total variation of the restriction of
$f$ to the interval $[0,t]$. Define
$\tilde{f}:[0,V_f/s]\longto X_{G}$ by:
$$\tilde{f}(t)=f\left(V_f^{-1}(st)\right), \qquad 0\le t\le
V_f/s.$$
In fact, as $V_f(t)$ is non-decreasing as a function of $t$ but
may perhaps be not strictly increasing, $V_f^{-1}(st)$ may be
undefined. However, the value of $f$ at all points up to which
the variation of $f$ assumes some value is the same, so that
$\tilde{f}$ is indeed well defined.

Now we claim that $\tilde{f}$ is also $s$-Lipschitz, and
therefore a possible strategy for the cop. In fact, for
$0\le t_1<t_2\le V_f/s$, consider the distance in $X_G$
between the points $\tilde{f}(t_1)=f\left(V_f^{-1}(st_1)\right)$ and
$\tilde{f}(t_2)=f\left(V_f^{-1}(st_2)\right)$. This is the distance between the
points on the cop's path to which he arrived after
traveling accumulated distances of $st_1$ and $st_2$,
respectively. Hence the distance between these two points is
at most $s(t_2-t_1)$.

It remains to show that $\tilde{f}$ is a winning strategy. Let
$\tilde{g}$ be any strategy of the robber, namely a
$1$-Lipschitz function from $[0,V_f/s]$ to $X_G$. Define another
strategy $g$ for the robber by:
$$g(t)=\tilde{g}(V_f(t)/s), \qquad 0\le t\le T_0.$$
This is also an admissible strategy for the robber since
for $0\le t_1<t_2\le T_0$:
$$\begin{array}{lll}
    d(g(t_1),g(t_2)) & = & d(\tilde{g}(V_f(t_1)/s),
\tilde{g}(V_f(t_2)/s)\\
                     & \le & \left|V_f(t_2)/s-V_f(t_1)/s\right|
                       =\frac{1}{s}\left|V_f(t_2)-V_f(t_1)\right|\\
                     & \le & \frac{1}{s} (st_2-st_1)=t_2-t_1.
  \end{array}$$
Since $f$ is a winning strategy, there exists some $t_0\in
[0,T_0]$ for which $f(t_0)=g(t_0)$. Hence,
$$\tilde{f}(V_f(t_0)/s)=f\left(V_f^{-1}\left(V_f(t_0)/s\right)\cdot
s\right)=f(t_0)=g(t_0)=\tilde{g}(V_f(t_0)/s),$$
namely, $\tilde{f}$ intersects $\tilde{g}$.
\end{Proof}

\bibliographystyle{plainurl}

\bigskip

\end{document}